\documentclass[11pt,a4paper]{article}
\usepackage{epsf,epsfig,amsfonts,amsgen,amsmath,amstext,amsbsy,amsopn,amsthm,cases,listings,color
}
\usepackage{ebezier,eepic}
\usepackage{color}
\usepackage{multirow}
\usepackage{epstopdf}
\usepackage{graphicx}
\usepackage{pgf,tikz}
\usepackage{mathrsfs}
\usepackage{enumitem}
\usepackage{booktabs}
\usepackage{url}
\usepackage{amssymb}
\usepackage{subfigure}
\usepackage{stmaryrd}

\usepackage{pgfplots}
\usepackage{subfigure}
\usepackage{mathrsfs}
\usetikzlibrary{arrows}

\allowdisplaybreaks[1]

\definecolor{uuuuuu}{rgb}{0.27,0.27,0.27}
\definecolor{sqsqsq}{rgb}{0.1255,0.1255,0.1255}

\newtheorem{definition}{Definition} [section]
\newtheorem{theorem}[definition]{Theorem}
\newtheorem{lemma}[definition]{Lemma}
\newtheorem{proposition}[definition]{Proposition}

\newtheorem*{claim*}{Claim}
\newtheorem{claim}[definition]{Claim}

\newtheorem{fact}[definition]{Fact}
\newenvironment{pf}{\noindent{\bf Proof}}

\begin{document}
\title{\bf\Large Tur\'{a}n problems in pseudorandom graphs}

\date{\today}

\author{
Xizhi Liu
\thanks{Mathematics Institute and DIMAP,
            University of Warwick,
            Coventry, CV4 7AL, UK.
Email: xizhi.liu@warwick.ac.uk.
Research was supported by ERC Advanced Grant 101020255.
}
\and
Dhruv Mubayi
\thanks{Department of Mathematics, Statistics, and Computer Science, University of Illinois, Chicago, IL, 60607 USA.
Email: mubayi@uic.edu.
Research partially supported by NSF awards DMS-1763317, 1952767, 2153576 and a Humboldt Research Award.}
\and
David Munh\'{a} Correia
\thanks{Department of Mathematics, ETH, 8092 Z\"{u}rich, Switzerland.
\newline
Email: david.munhacanascorreia@math.ethz.ch.}
}
\maketitle
%\footnote{footnote}
%%%%%%%%%%%%%%%%%%%%%%%%%%%%%%%%%%%%%%%%%%%%%%%%%
\begin{abstract}
Given a graph $F$, we consider the problem of determining the densest possible  pseudorandom graph that contains no copy of $F$. We provide an embedding procedure that improves a general result of Conlon, Fox, and Zhao which gives an upper bound on the density.
In particular, our result implies that optimally pseudorandom graphs with density greater than $n^{-1/3}$ must contain a copy of the Peterson graph, while the previous best result gives the bound $n^{-1/4}$. Moreover, we conjecture that the exponent $1/3$ in our bound is tight.
We also construct the densest known pseudorandom $K_{2,3}$-free graphs that are also triangle-free.
Finally, we obtain the densest known construction of clique-free pseudorandom graphs due to Bishnoi, Ihringer and Pepe in a novel way and give a different proof that they have no large clique.
\end{abstract}
%%%%%%%%%%%%%%%%%%%%%%%%%%%%%%%%%%%%%%%%%%%%%%%%%
\section{Introduction}\label{Pseudorandom:SEC:Introduction}
Given a family $\mathcal{F}$ of graphs we say a graph $G$ is $\mathcal{F}$-free if it does not contain any member in $\mathcal{F}$ as a subgraph.
A fundamental problem in extremal graph theory is to determine the maximum number $\mathrm{ex}(n, \mathcal{F})$ of edges in an $\mathcal{F}$-free graph on $n$ vertices.
Here $\mathrm{ex}(n, \mathcal{F})$ is called the \textit{Tur\'{a}n number} of $\mathcal{F}$, and the limit $\pi(F) = \lim_{n\to \infty} \mathrm{ex}(n, \mathcal{F})/\binom{n}{2}$, whose existence was proved by Katona, Nemetz, and Simonovits~\cite{KNS64}, is called the \textit{Tur\'{a}n density} of $\mathcal{F}$.

For a graph $G$ we use $V(G)$ to denote the vertex set of $G$, and use $v(G)$ and $e(G)$ to denote the number of vertices and edges in $G$, respectively.
For a set $S \subset V(G)$ we use $e_{G}(S)$ to denote the number of edges in the induced subgraph $G[S]$.
Given two vertex sets $X, Y \subset V(G)$ we use $e_{G}(X,Y)$ to denote the number of edges in $G$ that have one vertex in $X$ and one vertex in $Y$
(here edges with both vertices in $X\cap Y$ are counted twice, hence $e_{G}(X,X) = 2e_{G}(X)$).
We will omit the subscript $G$ if it is clear from the context.

Informally, we say that a graph is pseudorandom if its edge distribution behaves like a random graph.
In this note we use the following notation, which was firstly introduced by Thomason in his fundamental papers~\cite{Thomason87a, Thomason87b}, to quantify the randomness of a graph.

%\begin{definition}\label{DFN:jumbled}
For two real numbers $p\in [0,1]$ and $\alpha \ge 0$, we say a graph $G$ is \textit{$(p,\alpha)$-jumbled} if it satisfies
\begin{align}\label{equ:jumble}
\left|e(X,Y)-p|X||Y|\right| \le \alpha\sqrt{|X||Y|}
\end{align}
for all $X,Y\subset V(G)$.
%\end{definition}

A special family of $(p,\alpha)$-jumbled graphs are the well-known \textit{$(n,d,\lambda)$-graphs}.
A graph $G$ is an $(n,d,\lambda)$-graph if it is a $d$-regular graph on $n$ vertices and the second largest eigenvalue in absolute value of
its adjacency matrix is $\lambda$.
The well-known Expander mixing lemma (e.g. see~\cite[Theorem 2.11]{KS06}) implies that an $(n,d,\lambda)$-graph is $(d/n, \lambda)$-jumbled.
Conversely, Bilu and Linial~\cite{BL06} proved that an $n$-vertex $d$-regular $(p,\alpha)$-jumbled graph is an $(n, d, \lambda)$-graph
with $\lambda = O(\alpha \log(d/\alpha))$.

It is known that a random graph $G(n,p)$ is almost surely a $(p,\alpha)$-jumbled graph with $\alpha = O(\sqrt{np})$ (see e.g.~\cite[Corollary~2.3]{KS06}).
The proof of Erd\H{o}s and Spencer in~\cite{ES71} can be extended to show that every $(p,\alpha)$-jumbled graph on $n$ vertices satisfies that $\alpha = \Omega(\sqrt{np})$ (see e.g.~\cite{BS06,KS06}), and,
in particular, $\lambda = \Omega(\sqrt{d})$ for an $(n,d,\lambda)$-graph.
Therefore, an $n$-vertex $(p,\alpha)$-jumbled graph with $\alpha = \Theta(\sqrt{np})$ can be viewed as optimally pseudorandom.
The tightness of the bound $\lambda = \Omega(\sqrt{d})$ in general is also witnessed by many well-known explicit constructions.
For example, the well-known triangle-free $(n,d,\lambda)$-graph constructed by Alon~\cite{Alon94} satisfies  $d = \Theta(n^{2/3})$ and $\lambda = O(\sqrt{d})$.

Constructions of dense pseudorandom graphs that avoid a certain graph as a subgraph are extremely useful for many problems.
In particular, the second author and Versta\"{e}te~\cite{MV19} recently showed that for every fixed integer $t\ge 3$,
the existence of $K_t$-free $(n,d,\lambda)$-graphs with $d = \Omega(n^{1-\frac{1}{2t-3}})$ and $\lambda = O(\sqrt{d})$
implies the lower bound $R(t,n) = \Omega^{\ast}(n^{t-1})$ for the off-diagonal Ramsey numbers, and this matches the best known upper bound in exponent.
More generally,~\cite{MV19} shows that the existence of dense $F$-free pseudorandom graphs implies a good lower bound for the Ramsey number $R(F, n)$.
This motivates us to consider the following pseudorandom version of the Tur\'{a}n problem.

Let $\mathcal{F}$ be a family of graphs and $C>0$ be a real number.
Let $\mathrm{ex}_{\mathrm{rand}}(n,C,\mathcal{F})$ be the maximum number of edges in an $n$-vertex $(p,\alpha)$-jumbled $\mathcal{F}$-free graph with $\alpha \le C\sqrt{np}$.
Note that in the definition of $\mathrm{ex}_{\mathrm{rand}}(n,C,\mathcal{F})$ we do not have any restriction on $p$.

In many applications, it suffices to know the exponent of $\mathrm{ex}_{\mathrm{rand}}(n,C,\mathcal{F})$. So we let
\begin{align*}
    \mathrm{exp}(\mathcal{F})
    = \lim_{C\to \infty}\limsup_{n\to \infty} \frac{\log \left(\mathrm{ex}_{\mathrm{rand}}(n,C,\mathcal{F})/n\right)}{\log n}.
\end{align*}
In other words, \textit{$\mathrm{exp}(\mathcal{F})$} is the supremum of $\beta$
such that there exist a constant $C$ and a sequence $\left(G_n\right)_{n=1}^{\infty}$ of $\mathcal{F}$-free $(p_n,\alpha_n)$-jumbled graphs with
\begin{align}
\lim_{n\to \infty}v(G_n) = \infty, \quad
\lim_{n\to \infty}\frac{\log (p_n v(G_n))}{\log v(G_n)} \ge \beta, \quad \text{and}\quad
\alpha_n \le C \sqrt{p_n v(G_n)}. \notag
\end{align}

Using the Expander mixing lemma one can prove that for every integer $t\ge 3$
we have $\mathrm{exp}(K_t) \le 1 - \frac{1}{2t-3}$.
Alon's construction~\cite{Alon94} shows that this bound is tight for $t=3$, that is, $\mathrm{exp}(K_3) = \frac{2}{3}$.
It is a major open problem to determine $\mathrm{exp}(K_t)$ in general.
Alon and Krivelevich proved in~\cite{AK97} that $\mathrm{exp}(K_t) \ge 1- \frac{1}{t}$.
Recently, Bishnoi, Ihringer, and Pepe~\cite{BIP20} improved their bound and proved the following result.

\begin{theorem}[Bishnoi--Ihringer--Pepe~\cite{BIP20}]\label{THM:BIP20}
Suppose that $t\ge 4$ is an integer.
Then  $\mathrm{exp}(K_t) \ge 1- \frac{1}{t-1}$.
\end{theorem}

Mattheus and Pavese~\cite{MP22} give a different construction of $K_{t}$-free pseudorandom graphs which also matches the bound in Theorem~\ref{THM:BIP20}.
In Section~\ref{SEC:Kt-free}, we will present a construction that is isomorphic to the construction of Bishnoi, Ihringer, and Pepe~\cite{BIP20}, and give a new proof
to Theorem~\ref{THM:BIP20}.

For bipartite graphs, the pseudorandom version of the Tur\'{a}n problem does not appear to differ much from the ordinary Tur\'{a}n problem  since many constructions for the lower bound are pesudorandom. For example,
for  complete bipartite graphs, the projective norm graphs (see~\cite{KRS96,ARS99}) are optimally pseudorandom (see~\cite{Szabo03}) and do not contain $K_{s,t}$ with $t\ge (s-1)!+1$. Therefore, together with the well known K\"{o}vari--S\'{o}s--Tur\'{a}n Theorem~\cite{KST54},
we know that $\mathrm{exp}(K_{s,t}) = 1-\frac{1}{s}$ for all positive integers $s,t$ with $t\ge (s-1)!+1$.
For even cycles, constructions from generalized polygons~\cite{LUW99} and an old result of Bondy and Simonovits~\cite{BS74} imply that $\mathrm{exp}(C_6) = \frac{4}{3}$ and $\mathrm{exp}(C_{10}) = \frac{6}{5}$. The value of $\mathrm{exp}(C_{2k})$ for $k \neq 2,3,5$ are still unknown due to the lack of constructions.
For non-bipartite graphs, $\mathrm{exp}(F)$ is completely different from the ordinary Tur\'{a}n problem (indeed, $\mathrm{exp}(F)<2$ while $\pi(F) >0$)  and there are very  graphs $F$ for which $\mathrm{exp}(F)$ is known.
For example, for odd cycles,
a construction due to Alon and Kahale~\cite{AK98} together with Proposition~4.12 in~\cite{KS06} implies
that $\mathrm{exp}(C_{\ell}) = \frac{2}{\ell}$ for all odd integers $\ell \ge 3$.

The first general upper bound on $\exp(F)$,  is due to Kohayakawa R\"{o}dl,  Schacht, Sissokho, and Skokan~\cite{KRSSS07}. They prove that $\mathrm{exp}(F) \le 1 - \frac{1}{2\nu(F)-1}$
%{\color{blue} It seems that I made a mistake here, maybe the correct bound is $\mathrm{exp}(F) \le 1 - \frac{1}{2\nu(F)-1}$.}
for every triangle-free graph $F$.
Here $\nu(F) = \frac{1}{2}\left(d(F) + D(F) +1\right)$, where $D(F) = \min\left\{2d(F), \Delta(F)\right\}$ and $d(F)$ is the degeneracy of $F$. This was improved via the following result of Conlon, Fox, and Zhao in~\cite{CFZ14}.

\begin{theorem} [Conlon--Fox--Zhao~\cite{CFZ14}] \label{THM:CFZ}
For every graph $F$, we have $\mathrm{exp}(F) \le 1-\frac{1}{2d_2(F)+1}$,
where $d_2(F)$ is the minimum real number $d$ such that there is an ordering of the vertices $v_1, \ldots, v_m$ of $F$ so that $N_{<i}(v_i)+N_{<i}(v_j) \le 2d$ for all edges $v_i v_j \in F$.
Here $N_{<i}(v)$ is the number of neighbors of $v$ in $\{v_1, \ldots, v_{i-1}\}$.
%is the $2$-degeneracy of $F$ whose definition can be found in~\cite{CFZ14}.
\end{theorem}

In some cases the bound provided by Theorem~\ref{THM:CFZ} is sharp (e.g. it is conjectured to be sharp for cliques), but we speculate that for most graphs $F$ it can be improved. Below we give an improvement that holds for many graphs.

\begin{theorem}\label{THM:CFZ+tree}
Let $F$ be a fixed graph and $d$ be such that the following holds. There exists an ordering $v_1,v_2, \ldots, v_m$ of the vertices of $F$ and a $1 \leq k \leq m$ such that:
\begin{itemize}
    \item $F[\{v_{k}, \ldots, v_m\}]$ is a forest,
    \item for all edges $v_iv_j \in F$ with $i<j$, $i<k$, we have $N_{< i}(v_i) + N_{< i}(v_j) \leq 2d$, and
    \item for all edges $v_iv_j \in F$ with $k \leq i < j$, we have $N_{<k}(v_i) + N_{<k}(v_j) \leq 2d$.
\end{itemize}
Then, $\mathrm{exp}(F) \leq 1 - \frac{1}{2d+1}$.
\end{theorem}
\bigskip

\textbf{Remark.} Roughly speaking, Theorem~\ref{THM:CFZ+tree}  says that if there exists an induced forest on an interval of some ordering of $V(F)$, then it is possible to improve the bound of Theorem~\ref{THM:CFZ} by treating this forest as one edge.
In fact, we will see later that the proof of Theorem~\ref{THM:CFZ+tree} can be extended easily to get a  more general result.
\bigskip

As an application of Theorem~\ref{THM:CFZ+tree}, we study  $\mathrm{exp}(\textbf{P})$, where $\textbf{P}$ is the Petersen graph.
The Petersen graph was considered by several researchers in related contexts.
For example, Tait and Timmons~\cite{TT16} proved that the Erd\H{o}s--R\'{e}nyi orthogonal polarity graphs~\cite{ER62} (henceforth the Erd\H{o}s--R\'{e}nyi graph), which are optimally pseudorandom $C_4$-free graphs, contain the Petersen graph as a subgraph.
Conlon, Fox, Sudakov, and Zhao asked in~\cite{CFSZ21} whether there is a counting lemma for the Petersen graph in an $n$-vertex $C_4$-free graph with $\Omega(n^{3/2})$ edges.
We know very little about $\mathrm{exp}(\textbf{P})$,
for example,  it is not known whether  $\mathrm{exp}(\textbf{P}) \ge \frac{1}{2}$. The only lower bound we have  is  $\mathrm{exp}(\textbf{P}) \ge \mathrm{exp}(C_5) = \frac{2}{5}$. In the other direction, the previous best  upper bound is $\mathrm{exp}(\textbf{P}) \le \frac{3}{4}$ that follows from~\cite{CFZ14} (it is not too difficult to prove that $d_2(\textbf{P}) = \frac{3}{2}$).

%%%%%%%%%%%%%%%%%%%%%%%%%%%%%%%%%%%%%%%%%%%%%%
%%%%%%%%%%%%%%%%%%%%%%%%%%%%%%%%%%%%%%%%%%%%%%%%%%%%%%%
\begin{figure}[htbp]\label{Fig:Petersen}
\centering
\subfigure
{
\begin{minipage}[t]{0.45\linewidth}
\centering
\tikzset{every picture/.style={line width=1pt}} %set default line width to 0.75pt

\begin{tikzpicture}[x=0.75pt,y=0.75pt,yscale=-1,xscale=1]
%uncomment if require: \path (0,374); %set diagram left start at 0, and has height of 374

%Shape: Regular Polygon [id:dp6543675591941622]
\draw   (411.63,152.47) -- (377.17,259.15) -- (265.07,259.35) -- (230.24,152.79) -- (320.82,86.73) -- cycle ;
%Straight Lines [id:da9915029263040878]
\draw    (320.82,86.73) -- (320.6,141.28) ;
%Straight Lines [id:da849229020734138]
\draw    (411.63,152.47) -- (359.68,169.12) ;
%Straight Lines [id:da5226098520293585]
\draw    (377.17,259.15) -- (345.29,214.89) ;
%Straight Lines [id:da24122578213793888]
\draw    (297.31,215.34) -- (265.07,259.35) ;
%Straight Lines [id:da7665628200182351]
\draw    (230.24,152.79) -- (282.05,169.85) ;
%Straight Lines [id:da13579134606824694]
\draw    (320.6,141.28) -- (345.29,214.89) ;
%Straight Lines [id:da9066719542980706]
\draw    (282.05,169.85) -- (345.29,214.89) ;
%Straight Lines [id:da32328834696973585]
\draw    (359.68,169.12) -- (282.05,169.85) ;
%Straight Lines [id:da8600078238632325]
\draw    (297.31,215.34) -- (359.68,169.12) ;
%Straight Lines [id:da10166710359426867]
\draw    (320.6,141.28) -- (297.31,215.34) ;
%%%%%%%%%%%%%%%%%%%%%%%%%%%%%%%%%%%%%%%%%%%%%%%%%%%%%%
\draw[line width=2pt] (411.63,152.47) circle [radius = 1.5];
\draw[line width=2pt] (377.17,259.15) circle [radius = 1.5];
\draw[line width=2pt] (265.07,259.35) circle [radius = 1.5];
\draw[line width=2pt] (230.24,152.79) circle [radius = 1.5];
\draw[line width=2pt] (320.82,86.73) circle [radius = 1.5];
\draw[line width=2pt] (320.6,141.28) circle [radius = 1.5];
\draw[line width=2pt] (359.68,169.12) circle [radius = 1.5];
\draw[line width=2pt] (345.29,214.89) circle [radius = 1.5];
\draw[line width=2pt] (265.07,259.35) circle [radius = 1.5];
\draw[line width=2pt] (282.05,169.85) circle [radius = 1.5];
\draw[line width=2pt] (297.31,215.34) circle [radius = 1.5];
%%%%%%%%%%%%%%%%%%%%%%%%%%%%%%%%%%%%%%%
% Text Node
\draw (318,68) node [anchor=north west][inner sep=0.75pt]   [align=left] {1};
% Text Node
\draw (418,147) node [anchor=north west][inner sep=0.75pt]   [align=left] {4};
% Text Node
\draw (213,147) node [anchor=north west][inner sep=0.75pt]   [align=left] {2};
% Text Node
\draw (247,263) node [anchor=north west][inner sep=0.75pt]   [align=left] {5};
% Text Node
\draw (379,263) node [anchor=north west][inner sep=0.75pt]   [align=left] {10};
% Text Node
\draw (328,132) node [anchor=north west][inner sep=0.75pt]   [align=left] {3};
% Text Node
\draw (335,219) node [anchor=north west][inner sep=0.75pt]   [align=left] {8};
% Text Node
\draw (299,219) node [anchor=north west][inner sep=0.75pt]   [align=left] {7};
% Text Node
\draw (272,174) node [anchor=north west][inner sep=0.75pt]   [align=left] {6};
% Text Node
\draw (361,174) node [anchor=north west][inner sep=0.75pt]   [align=left] {9};
\end{tikzpicture}
\end{minipage}
}
%%%%%%%%%%%%%%%%%%%%%%%%%%%%%%%%%%%%%%
\subfigure
{
\begin{minipage}[t]{0.45\linewidth}
\centering
\tikzset{every picture/.style={line width=1pt}} %set default line width to 0.75pt
\begin{tikzpicture}[x=0.75pt,y=0.75pt,yscale=-1.1,xscale=1.1]
%uncomment if require: \path (0,374); %set diagram left start at 0, and has height of 374

%Straight Lines [id:da2532248342625114]
\draw    (140,170) -- (200,120) ;
%Straight Lines [id:da7427449710126133]
\draw    (140,170) -- (200,170) ;
%Straight Lines [id:da2147803345828203]
\draw    (140,170) -- (200,220) ;
%Curve Lines [id:da9335459880804347]
\draw    (200,120) -- (300,100) ;
%Curve Lines [id:da12090079952608956]
\draw    (200,120) -- (300,170) ;
%Curve Lines [id:da8833524804246176]
\draw    (200,170) -- (300,240) ;
%Curve Lines [id:da3608938856371524]
\draw    (200,170) -- (350,240) ;
%Curve Lines [id:da4513665571020591]
\draw   (200,220) -- (350,170) ;
%Curve Lines [id:da5329008401556488]
\draw   (200,220) -- (350,100) ;
%Straight Lines [id:da2442194483492881]
\draw    (300,100) -- (350,100) ;
%Straight Lines [id:da7850415101963177]
\draw    (300,170) -- (350,240) ;
%Straight Lines [id:da0313269194048067]
\draw    (350,170) -- (300,240) ;
%Straight Lines [id:da6377283050760938]
\draw    (300,170) -- (350,170) ;
%Curve Lines [id:da37086328727966555]
\draw    (300,100) .. controls (280,160) and (280,180) .. (300,240) ;
%Curve Lines [id:da37825718826157684]
\draw    (350,100) .. controls (370,160) and (370,180) .. (350,240) ;
%%%%%%%%%%%%%%%%%%%%%%%%%%%%%%%%%%%%%%%
\draw[line width=2pt] (140,170)  circle [radius = 1.5];
\draw[line width=2pt] (200,120)  circle [radius = 1.5];
\draw[line width=2pt] (200,170)  circle [radius = 1.5];
\draw[line width=2pt] (200,220)  circle [radius = 1.5];
\draw[line width=2pt] (300,100)  circle [radius = 1.5];
\draw[line width=2pt] (300,170)  circle [radius = 1.5];
\draw[line width=2pt] (300,240) circle [radius = 1.5];
\draw[line width=2pt] (350,100)  circle [radius = 1.5];
\draw[line width=2pt] (350,170)  circle [radius = 1.5];
\draw[line width=2pt] (350,240)  circle [radius = 1.5];
%%%%%%%%%%%%%%%%%%%%%%%%%%%%%%%%%%%%%%%
% Text Node
\draw (120,165) node [anchor = north west][inner sep=0.75pt]   [align=left] {$1$};
% Text Node
\draw (195,150) node [anchor=north west][inner sep=0.75pt]   [align=left] {$3$};
% Text Node
\draw (195,100) node [anchor=north west][inner sep=0.75pt]   [align=left] {$2$};
% Text Node
\draw (195,200) node [anchor=north west][inner sep=0.75pt]   [align=left] {$4$};
% Text Node
\draw (295,80) node [anchor=north west][inner sep=0.75pt]   [align=left] {$6$};
% Text Node
\draw (345,80) node [anchor=north west][inner sep=0.75pt]   [align=left] {$9$};
% Text Node
\draw (295,150) node [anchor=north west][inner sep=0.75pt]   [align=left] {$5$};
% Text Node
\draw (340,150) node [anchor=north west][inner sep=0.75pt]   [align=left] {$10$};
% Text Node
\draw (295,248) node [anchor=north west][inner sep=0.75pt]   [align=left] {$8$};
% Text Node
\draw (345,248) node [anchor=north west][inner sep=0.75pt]   [align=left] {$7$};
\end{tikzpicture}
\end{minipage}
}
\caption{The Petersen graph in two different drawings.}
\end{figure}
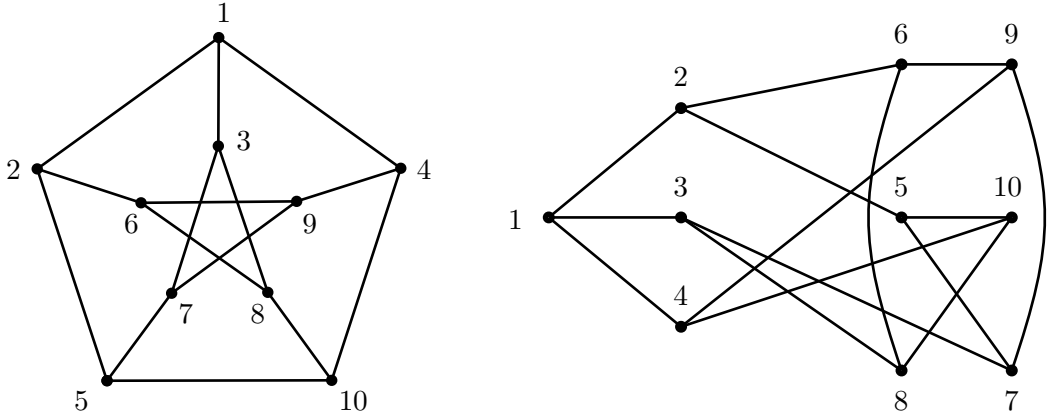
%%%%%%%%%%%%%%%%%%%%%%%%%%%%%%%%%%%%%%%%%%%%%%%%%%%%%%%%%

\begin{theorem}\label{THM:Petersen-exp}
We have $\mathrm{exp}(\textbf{P}) \le \frac{2}{3}$.
\end{theorem}

We conjecture that $\mathrm{exp}(\textbf{P}) = \frac{2}{3}$.
We think that the construction of $K_3$-free pseudorandom graphs due to Kopparty does not contain the Petersen graph as a subgraph. If this is true, then it will prove the lower bound $\mathrm{exp}(\textbf{P}) \ge \frac{2}{3}$. For completeness, we include his construction here.

Let $p\neq 3$ be a prime, and let $\mathbb{F}_{q}$ be a finite field with
where $q = p^h$ for some integer $h\ge 1$.
% Let $T\subset \mathbb{F}_{q}$ be a set that contains no solution (over $\mathbb{F}_{q}$) to the equation $x+y+z=0$.
Recall that the \textit{absolute trace function} $\mathrm{Tr}\colon \mathbb{F}_{q} \to \mathbb{F}_{p}$ is defined as $\mathrm{Tr}(\alpha) = \alpha + \alpha^{p} + \cdots + \alpha^{p^{h-1}}$ for every $\alpha \in \mathbb{F}_{q}$.

Let $V = \mathbb{F}_{q}^{3}$, $T =  \left\{x \in \mathbb{F}_{q} \colon \mathrm{Tr}(x) \in \{1, -1\}\right\}$, and $S\subset \mathbb{F}_{q}^3$ be a subset defined as
\begin{align*}
    S = \left\{(xy, xy^2, xy^3) \colon x\in T, \ y\in \mathbb{F}_{q}\setminus\{0\}\right\}.
\end{align*}
Kopparty's construction is the graph $G$ on $V$ in which two vertices $\textbf{u}, \textbf{v} \in V$ are adjacent iff $\textbf{u}-\textbf{v} \in S$.
Using some simple linear algebra  one can show that $G$ is triangle-free,
and using some results about finite fields and abelian groups one can prove that $G$ is an $(n,d,\lambda)$-graph with $n = q^3$, $d = \Theta(\frac{q^2}{p})$, and $\lambda = \Theta(\frac{q}{p})$.

%%%%%%%%%%%%%%%%%%%%%%%%%%%%%%%%%%%%%%%%%%%%
\textbf{Remark.}
Ferdinand Ihringer informed us that the construction above contains an induced copy of the Petersen graph when $p=2$ and $h=3$,
and he thinks that, in general, Kopparty's construction contains many copies of the Petersen graph.
Nevertheless, it still might be true that $\mathrm{exp}(\textbf{P}) = \frac{2}{3}$.
%%%%%%%%%%%%%%%%%%%%%%%%%%%%%%%%%%%%%%%%%%%%

Our next result about $K_{2,3}$ was motivated by an old problem of Erd\H{o}s~\cite{Erdos75},
which asks if
\begin{align*}
    \mathrm{ex}(n,\{K_3, C_4\}) = \left(\frac{1}{2\sqrt{2}}+o(1)\right)n^{3/2}
\end{align*}
is true.
A construction due to Parsons~\cite{Parson76} for the lower bound comes from the Erd\H{o}s--R\'{e}nyi graph by removing half of its vertices.
Since the Erd\H{o}s--R\'{e}nyi graph is optimally pseudorandom, Parsons' construction also implies that
$\mathrm{ex}_{\mathrm{rand}}(n,C,\{K_3, C_4\}) \ge \left(\frac{1}{2\sqrt{2}}+o(1)\right)n^{3/2}$ for some absolute constant $C$.
In~\cite{AKSV14}, Allen, Keevash, Sudakov, and Verstra\"{e}te proved that the extremal constructions for $\mathrm{ex}(n,\{K_3, K_{2,t}\})$ cannot be bipartite for every $t\ge 3$
by constructing a $\{K_3, K_{2,t}\}$-free graph whose number of edges is greater than the maximum number of edges in a $\{K_3, K_{2,t}\}$-free bipartite graph.
However, their construction is $(t-1)$-partite, and therefore it does not give a lower bound for $\mathrm{ex}_{\mathrm{rand}}(n,C,\{K_3, K_{2,3}\})$.  The previous best lower bound is $$\mathrm{ex}_{\mathrm{rand}}(n,C,\{K_3, K_{2,3}\})\ge
\mathrm{ex}_{\mathrm{rand}}(n,C,\{K_3, C_4\})\ge \left(\frac{1}{2\sqrt 2}-o(1)\right) n^{3/2}$$
that follows from Parsons' construction.
We improve this and present a construction of the densest known $\{K_3, K_{2,3}\}$-free pseudorandom graphs.

\begin{theorem}\label{THM:K3-K2,3}
We have $\mathrm{ex}_{\mathrm{rand}}(n,2,\{K_3, K_{2,3}\}) \ge \left(\frac{1}{2}-o(1)\right)n^{3/2}$.
\end{theorem}

\textbf{Remark.}
Thang Pham pointed out to us that the following construction also provides a lower bound for $\mathrm{ex}_{\mathrm{rand}}(n,2,\{K_3, K_{2,3}\})$.
Let $q$ be an odd prime power.
The distance graph $D$ on $\mathbb{F}_{q}^2$ is a graph whose vertex set is  $\mathbb{F}_{q}^2$,
and two points $(x_1,x_2), (y_1, y_2)$ are adjacent iff $(x_1-y_1)^2+(x_2-y_2)^2=1$.
The $K_{2,3}$-freeness of $D$ follows from the fact that any two cycles have at most two points in the intersection.
The $K_3$-freeness of $D$ follows from results in~\cite{BIP14}.
The pseudorandomness of $D$ follows from results in~\cite{IR07}.

In Section~\ref{SEC:Proof-Petersen}, we prove Theorems~\ref{THM:CFZ+tree} and \ref{THM:Petersen-exp}.
In Section~\ref{SEC:K2-K23}, we prove Theorem~\ref{THM:K3-K2,3}.
In Section~\ref{SEC:Kt-free} we present a new proof of   Theorem~\ref{THM:BIP20}. Throughout the paper we will omit the use of floors and ceilings to make the presentation cleaner.
%%%%%%%%%%%%%%%%%%%%%%%%%%%%%%%%%%%%%%%%%%%%%%%%%%%%%%%%%%
\section{Upper bound for the Petersen graph}\label{SEC:Proof-Petersen}
\subsection{Proof of Theorem~\ref{THM:Petersen-exp}}\label{SUBSEC:Petersen-proof}
We prove Theorem~\ref{THM:Petersen-exp} in this section. In the next section, we will show that Theorem~\ref{THM:Petersen-exp} follows immediately from the more general Theorem~\ref{THM:CFZ+tree}, but we think it is instructive to see an independent proof of Theorem~\ref{THM:Petersen-exp} first.

%The proofs are repeatedly application of Equation~(\ref{equ:jumble}).
Let us present first two standard lemmas. We start with the following direct consequence of the definition of a jumbled graph.
\begin{lemma}\label{lem:standard}
Fix a real number $q >1$.
Let $G$ be a $(p,\alpha)$-jumbled graph on $n$ vertices and let $X,Y_1, \ldots, Y_{m} \subseteq V(G)$ be pairwise disjoint subsets. If $|X||Y_i| \geq q^2m\left(\frac{\alpha}{p} \right)^2$ for all $i\in [m]$, then there exists a vertex $x \in X$ with at least $\frac{q-1}{q}p|Y_i|$ neighbours in $Y_i$ for all $i\in [m]$. In particular, $e(X,Y_i) > 0$ for all $i\in [m]$.
\end{lemma}
\begin{proof}
Define $X_i := \{v\in X \colon |N_{G}(v) \cap Y_i| < \frac{q-1}{q}p|Y_i|\}$ for $i\in [m]$.
It suffices to prove that $|X_i| < \frac{|X|}{m}$ for all $i\in [m]$.
Suppose to the contrary that $|X_i| \ge \frac{|X|}{m}$ for some $i\in [m]$.
By the definition of jumbleness, we get
\begin{align*}
    e(X_i,Y_i)
    \ge p|X_i||Y_i| - \alpha\sqrt{|X_i||Y_i|}
    = \left(p-\frac{\alpha}{\sqrt{|X_i||Y_i|}}\right)|X_i||Y_i|.
\end{align*}
It follows from $|X||Y_i| \geq q^2m\left(\frac{\alpha}{p} \right)^2$ that $\alpha \le \frac{p}{q} \sqrt{\frac{|X||Y_i|}{m}} \le \frac{p}{q} \sqrt{|X_i||Y_i|}$.
Therefore, it follows from the inequality above that
\begin{align*}
    e(X_i,Y_i)
    \ge \frac{q-1}{q} p |X_i||Y_i|,
\end{align*}
but our definition of $X_i$ yields $e(X_i,Y_i) < \frac{q-1}{q} p |X_i||Y_i|$, a contradiction.
%By averaging, there exists a vertex $x \in X$ with at least $p|Y|/2$ neighbours in $Y$.
\end{proof}
The next lemma is a simple cleaning procedure which is useful in problems concerning $(p,\alpha)$-jumbled graphs.

\begin{lemma}\label{lem:cleaning}
Let $G$ be a $(p,\alpha)$-jumbled graph on $n$ vertices. Then, for all sets $X,Y \subseteq V(G)$ such that $|X||Y| \geq 100 (\alpha / p)^2$ the following holds. There exist subsets $X' \subseteq X, Y' \subseteq Y$ respectively of size at least $9|X|/10$ and $9|Y|/10$ such that all $v \in X'$ have $d(v,Y') \geq p|Y'|/10$ and all $u \in Y'$ have $d(u,X') \geq p|X'|/10$.
\end{lemma}
\begin{proof}
Consider the following process. Start with $X_0 := X, Y_0 := Y$ and at step $i \geq 0$, do the following. Take $G_i := G[X_i,Y_i]$ and if there exists a vertex $v \in X_i$ such that $d(v,Y_i) < p|Y_i|/10$ or a vertex $v \in Y_i$ such that $d(v,X_i) < p|X_i|/10$, remove it from $X_i$, $Y_i$ respectively, giving new $X_{i+1},Y_{i+1}$. We claim that this process stops before $|X_i| \leq 9|X|/10$ or $|Y_i| \leq 9|Y|/10$, which would imply that we are done. Indeed, if it did not stop before that, consider the step at which, w.l.o.g., $|X_i| = 9|X|/10$ and $|Y_i| \ge 9|Y|/10$. By construction, every vertex in $X \setminus X_i$ has less than $p|Y|/10 \leq p|Y_i|/9$ neighbours in $Y_i$. So,
$$e(X \setminus X_i,Y_i) \leq p|Y_i||X \setminus X_i|/9 .$$
On the other hand, the definition of a $(p,\alpha)$-jumbled graph implies that
$$e(X \setminus X_i,Y_i) \geq p|Y_i||X \setminus X_i| - \alpha \sqrt{|Y_i||X \setminus X_i|}$$
which is a contradiction since $|X \setminus X_i||Y_i| \geq \frac{1}{10} \cdot \frac{9}{10} \cdot |X||Y| \geq 4 \left(\frac{\alpha}{p} \right)^2$.
\end{proof}

%%%%%%%%%%%%%%%%%%%%%%%%%%%%%%%%%%%%%%%%%
\begin{figure}[htbp]
\centering
\tikzset{every picture/.style={line width=1pt}} %set default line width to 0.75pt

\begin{tikzpicture}[x=0.75pt,y=0.75pt,yscale=-1,xscale=1]
%uncomment if require: \path (0,374); %set diagram left start at 0, and has height of 374

%Shape: Ellipse [id:dp9662606671743235]
\draw   (216.42,287.37) .. controls (171.91,287.57) and (135.61,238.79) .. (135.34,178.41) .. controls (135.06,118.04) and (170.92,68.93) .. (215.42,68.73) .. controls (259.93,68.52) and (296.23,117.3) .. (296.5,177.68) .. controls (296.78,238.06) and (260.92,287.16) .. (216.42,287.37) -- cycle ;
%Shape: Ellipse [id:dp3614246943183663]
\draw   (476.63,337.52) .. controls (413.47,337.81) and (361.95,267.76) .. (361.55,181.06) .. controls (361.16,94.37) and (412.05,23.86) .. (475.21,23.57) .. controls (538.37,23.29) and (589.89,93.33) .. (590.29,180.03) .. controls (590.68,266.72) and (539.8,337.23) .. (476.63,337.52) -- cycle ;
%Shape: Circle [id:dp46624419259934924]
\draw   (186,118) .. controls (186,104.19) and (197.19,93) .. (211,93) .. controls (224.81,93) and (236,104.19) .. (236,118) .. controls (236,131.81) and (224.81,143) .. (211,143) .. controls (197.19,143) and (186,131.81) .. (186,118) -- cycle ;
%Shape: Circle [id:dp14921646143532508]
\draw   (190,236) .. controls (190,222.19) and (201.19,211) .. (215,211) .. controls (228.81,211) and (240,222.19) .. (240,236) .. controls (240,249.81) and (228.81,261) .. (215,261) .. controls (201.19,261) and (190,249.81) .. (190,236) -- cycle ;
%Shape: Circle [id:dp2709605091749041]
\draw   (405,249) .. controls (405,235.19) and (416.19,224) .. (430,224) .. controls (443.81,224) and (455,235.19) .. (455,249) .. controls (455,262.81) and (443.81,274) .. (430,274) .. controls (416.19,274) and (405,262.81) .. (405,249) -- cycle ;
%Shape: Circle [id:dp13861803921594595]
\draw   (492,248) .. controls (492,234.19) and (503.19,223) .. (517,223) .. controls (530.81,223) and (542,234.19) .. (542,248) .. controls (542,261.81) and (530.81,273) .. (517,273) .. controls (503.19,273) and (492,261.81) .. (492,248) -- cycle ;
%Shape: Circle [id:dp6547136872680792]
\draw   (403,148) .. controls (403,134.19) and (414.19,123) .. (428,123) .. controls (441.81,123) and (453,134.19) .. (453,148) .. controls (453,161.81) and (441.81,173) .. (428,173) .. controls (414.19,173) and (403,161.81) .. (403,148) -- cycle ;
%Shape: Circle [id:dp9617457337443045]
\draw   (491,150) .. controls (491,136.19) and (502.19,125) .. (516,125) .. controls (529.81,125) and (541,136.19) .. (541,150) .. controls (541,163.81) and (529.81,175) .. (516,175) .. controls (502.19,175) and (491,163.81) .. (491,150) -- cycle ;
%Shape: Ellipse [id:dp23159980104946398]
\draw   (388,249.5) .. controls (388,224.37) and (426.28,204) .. (473.5,204) .. controls (520.72,204) and (559,224.37) .. (559,249.5) .. controls (559,274.63) and (520.72,295) .. (473.5,295) .. controls (426.28,295) and (388,274.63) .. (388,249.5) -- cycle ;
%Straight Lines [id:da7427449710126133]
\draw    (62,176) -- (215.92,178.05) ;
%Straight Lines [id:da2442194483492881]
\draw    (444,72) -- (504,72) ;
%Straight Lines [id:da2532248342625114]
\draw    (62,176) -- (211,118) ;
%Straight Lines [id:da2147803345828203]
\draw    (62,176) -- (215,236) ;
%Curve Lines [id:da9335459880804347]
\draw    (211,118) .. controls (237,69) and (395,38) .. (444,72) ;
%Curve Lines [id:da12090079952608956]
\draw    (211,118) .. controls (248,149) and (378,163) .. (428,148) ;
%Curve Lines [id:da4513665571020591]
\draw    (215,236) .. controls (228.08,210.95) and (451,91) .. (504,72) ;
%Curve Lines [id:da5329008401556488]
\draw   (215,236) .. controls (233.5,251.5) and (322,239) .. (380,216) .. controls (438,193) and (490,173) .. (516,150) ;
%Curve Lines [id:da8833524804246176]
\draw    (215.92,178.05) .. controls (259.92,192.05) and (380,264) .. (430,249) ;
%Curve Lines [id:da3608938856371524]
\draw    (215.92,178.05) .. controls (247,173) and (331,183) .. (375,192) .. controls (419,201) and (432.5,207.44) .. (464,219) .. controls (495.5,230.56) and (505,240) .. (517,248) ;
%Straight Lines [id:da7850415101963177]
\draw    (428,148) -- (516,150) ;
%Straight Lines [id:da0313269194048067]
\draw    (428,148) -- (517,248) ;
%Straight Lines [id:da6377283050760938]
\draw    (430,249) -- (516,150) ;
%Curve Lines [id:da37086328727966555]
\draw    (430,249) .. controls (380,221) and (377,90) .. (444,72) ;
%Curve Lines [id:da37825718826157684]
\draw    (517,248) .. controls (585,231) and (550,78) .. (504,72) ;
%%%%%%%%%%%%%%%%%%%%%%%%%%%%%%%%%%%%%%%
\draw[line width=2pt] (62,176)  circle [radius = 1.5];
\draw[line width=2pt] (215.92,178.05)  circle [radius = 1.5];
\draw[line width=2pt] (215,236)  circle [radius = 1.5];
\draw[line width=2pt] (211,118)  circle [radius = 1.5];
\draw[line width=2pt] (444,72)  circle [radius = 1.5];
\draw[line width=2pt] (428,148)  circle [radius = 1.5];
\draw[line width=2pt] (504,72) circle [radius = 1.5];
\draw[line width=2pt] (516,150)  circle [radius = 1.5];
\draw[line width=2pt] (430,249)  circle [radius = 1.5];
\draw[line width=2pt] (517,248)  circle [radius = 1.5];
%%%%%%%%%%%%%%%%%%%%%%%%%%%%%%%%%%%%%%%
% Text Node
\draw (40,170) node [anchor=north west][inner sep=0.75pt]   [align=left] {$v_1$};
% Text Node
\draw (204,184) node [anchor=north west][inner sep=0.75pt]   [align=left] {$v_3$};
% Text Node
\draw (201,145) node [anchor=north west][inner sep=0.75pt]   [align=left] {$Z_2$};
% Text Node
\draw (205,262) node [anchor=north west][inner sep=0.75pt]   [align=left] {$Z_4$};
% Text Node
\draw (436,55) node [anchor=north west][inner sep=0.75pt]   [align=left] {$v_6$};
% Text Node
\draw (494,55) node [anchor=north west][inner sep=0.75pt]   [align=left] {$v_9$};
% Text Node
\draw (416,174) node [anchor=north west][inner sep=0.75pt]   [align=left] {$Z_5$};
% Text Node
\draw (508,176) node [anchor=north west][inner sep=0.75pt]   [align=left] {$Z_{10}$};
% Text Node
\draw (466,302) node [anchor=north west][inner sep=0.75pt]   [align=left] {$Z_{7,8}$};
% Text Node
\draw (434,272) node [anchor=north west][inner sep=0.75pt]   [align=left] {$Z_8$};
% Text Node
\draw (495,272) node [anchor=north west][inner sep=0.75pt]   [align=left] {$Z_7$};

\end{tikzpicture}

\caption{An auxiliary picture for the proof of Theorem~\ref{THM:Petersen-jumbled-2nd}.}
\end{figure}
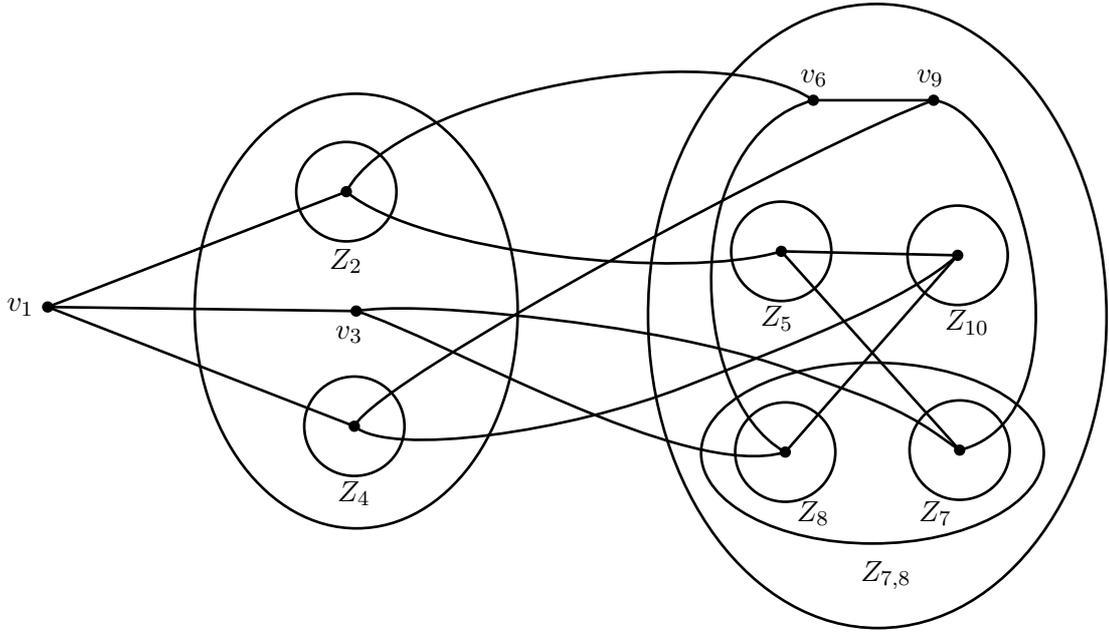

To prove Theorem~\ref{THM:Petersen-exp}, it suffices  to show that for every $C_1>0$ there exist $C_2>0$ and $n_0>0$ such that if $n>n_0$ and $G$ is an $n$-vertex graph that is $(p, \alpha)$-jumbled with $\alpha \le C_1\sqrt{p n}$ and $pn \ge C_2 n^{2/3}$, then $G$ contains the Peterson graph.
This follows from the following theorem.

\begin{theorem}\label{THM:Petersen-jumbled-2nd}
Let $G$ be a $(p,\alpha)$-jumbled graph on $n$ vertices such that $\alpha \leq p^2n/200$ and $p > 10n^{-1/3}$. If $n$ is sufficiently large, then $G$ contains the Petersen graph.
\end{theorem}

\textbf{Remark.}
Theorem~\ref{THM:Petersen-jumbled-2nd} and some simple calculations show that for every $C>0$ if $n$ is sufficiently large and $G$ is an $n$-vertex $(p,\alpha)$-jumbled graph with $\alpha \le C\sqrt{pn}$ and $pn \ge \left(200C+10\right)^{2/3} n^{2/3}$, then $G$ contains the Petersen graph as a subgraph.
%%%%%%%%%%%%
\begin{proof}
Let $G$ be a $(p,\alpha)$-jumbled graph on $n$ vertices such that $\alpha \leq p^2n/200$, so that $p^2n^2/5000 \ge 4(\alpha/p)^2$. We will find an embedding of the Petersen graph with vertices $v_1, \ldots, v_{10}$ which correspond to the labelling in the right drawing of Figure \ref{Fig:Petersen}. First, let $v_1$ be a vertex of $G$ of degree at least  $pn/2$ (which is guaranteed by Lemma \ref{lem:standard} with $q=2$) and let $X$ denote a set of $pn/2$ of its neighbours. Let also $Y$ denote the rest of the vertices, that is, $Y := V \setminus (\{v_1\} \cup X)$, which is of size at least $n/2$. By Lemma \ref{lem:cleaning}, there exist subsets $X' \subseteq X, Y' \subseteq Y$ of size at least $9pn/20$ and $9n/20$ respectively, with the properties as in the statement. Let $v_3 \in X'$ and consider its neighbourhood in $Y'$, which is guaranteed to be of size at least $9pn/200$, and take a subset $Z_{7,8} \subseteq Y'$ of it of size precisely $9pn/200$. Now, again applying Lemma \ref{lem:cleaning} with $Z_{7,8}$ and $Y' \setminus Z_{7,8}$ we have that there are at least $\frac{9}{10}|Y'\setminus Z_{7,8}| \ge \frac{9}{10}\left(\frac{9}{20}n-\frac{9}{200}pn\right)\ge \frac{2}{5}n$ vertices in $Y' \setminus Z_{7,8}$ with at least $9p|Z_{7,8}|/200 \geq 81p^2n/40000 > p^2n/500$ neighbours in $Z_{7,8}$. Let $v_6v_9$ be an edge contained in these $2n/5$ vertices, which is guaranteed by Lemma \ref{lem:standard}, so that both $v_6,v_9$ have at least $p^2n/500$ neighbours in $Z_{7,8}$. Let $Z_7$ be a set of $p^2n/1000$ such neighbours of $v_9$ and $Z_8$ be a set of $p^2n/1000$ such neighbours of $v_6$ so that $Z_7 \cap Z_8 = \emptyset$.

Now, recall that since $v_6,v_9 \in Y'$, both of them have at least $p|X'|/10 \geq 9p^2n/200 \ge p^2n/25$ neighbours in $X'$ and so, let $Z_2 \subseteq X'$ be a set of $p^2n/50$ such neighbours of $v_6$ and $Z_4$ be a set of $p^2n/50$ such neighbours of $v_9$ so that $Z_2 \cap Z_4 = \emptyset$.
Let $Y'' = Y' \setminus (Z_{7,8} \cup \{v_6,v_9\})$.
By Lemma \ref{lem:cleaning} applied to $Z_2,Y''$ and $Z_4,Y''$ there are disjoint subsets $Z_5,Z_{10} \subseteq Y''$ of size at least $\frac{1}{2}\frac{9}{10}|Y''|
\ge \frac{9}{20}\left(\frac{9}{20}n-\frac{9}{200}pn-2\right) \ge \frac{n}{5}$ such that every vertex in $Z_5$ has at least $\frac{9p}{100}\frac{p^2 n}{50} \ge 1$ neighbours in $Z_2$ and every vertex in $Z_{10}$ has at least $\frac{9p}{100}\frac{p^2 n}{50} \ge 1$ neighbours in $Z_4$.
Here we used that $p \ge 10n^{-1/3}$.

Furthermore, apply Lemma \ref{lem:standard} to $Z_7,Z_5$ and $Z_8, Z_{10}$ to find a vertex $v_7 \in Z_7$ with at least $p|Z_5|/2 \geq pn/10$ neighbours in $Z_5$ - let $Z'_5 \subseteq Z_5$ denote the set of neighbors of $v_7$ in $Z_5$; and a vertex $v_8 \in Z_8$ with at least $p|Z_{10}|/2 \geq pn/10$ neighbours in $Z_{10}$ - let $Z'_{10} \subseteq Z_{10}$ denote the set of neighbors of $v_8$ in $Z_{10}$. Recall now that $v_1v_3,v_7v_3,v_8v_3,v_7v_9,v_8v_6, v_6v_9$ are all edges. To finish, we note that if there exists an edge in $E[Z'_5,Z'_{10}]$, then the Petersen graph can be embedded. Indeed let $v_5v_{10}$ be such an edge with $v_5 \in Z'_5$ and $v_{10} \in Z'_{10}$. In particular, we have that $v_5v_7, v_{10}v_8$ are edges. Further, by the definition of $Z_{5},Z_{10}$, there exist $v_2 \in Z_2, v_4 \in Z_4$ such that $v_2v_5$ and $v_4v_{10}$ are edges. Furthermore, by definition, we also have that $v_2v_6,v_2v_1,v_4v_9,v_4v_1 $ are edges and thus, one can check that all the edges in the Petersen graph are present. To conclude then, note that there exists an edge in $E[Z'_5,Z'_{10}]$ by Lemma \ref{lem:standard} since $|Z'_5||Z'_{10}| \geq p^2n^2/100 > 4(\alpha/p)^2$.
\end{proof}

%%%%%%%%%%%%%%%%%%%%%%%%%%%%%%%%%%%%%%%%%%%%%%%%%%%%%%%
\subsection{Upper bound for general graphs}\label{SUBSEC:general-bound-CFZ}
In this section, we formalize our strategy used in the proof of Theorem~\ref{THM:Petersen-exp} by proving the more general Theorem~\ref{THM:CFZ+tree}. Before proceeding with the proof, we note  that Theorem~\ref{THM:Petersen-exp} was proved by embedding the vertices of $\textbf{P}$ in the order $(1,3,6,9,2,4,5,7,8,10)$. In fact, Theorem~\ref{THM:Petersen-exp}
follows from Theorem~\ref{THM:CFZ+tree} by letting the ordering of $V(\textbf{P})$ be $(v_1, \ldots, v_{10}) = (1,3,6,9,2,4,5,7,8,10)$ and choosing $k=5$.

Let us prove the following embedding lemma for forests first.

%%%%%%%%%%%%%%%%%%%%%%%%%%%%%%%%%%%%%%
\begin{lemma}\label{LEMMA:embed-tree}
    Suppose that $T$ is a forest on $[m]$ and $G$ is an $n$-vertex $(p,\alpha)$-jumbled graph.
    Let $X_1, \ldots, X_{m} \subset V(G)$ be nonempty pairwise disjoint subsets of $V(G)$ that satisfy
    \begin{align*}
        |X_i||X_j| \ge 2^m \left(\frac{\alpha}{p}\right)^2
    \end{align*}
    for all edges $ij$ in $T$.
    Then there exists an embedding of $f\colon T \to G$ such that $f(i) \in X_i$ for all $i\in [m]$.
\end{lemma}
\begin{proof}
    We prove this lemma by induction on $m$. The base case $m=1$ is clear since $X_1$ is nonempty. So we may assume that $m \ge 2$.
    Without loss of generality, we may assume that the vertex $m$ is a leaf of $T$ and the vertex $m-1$ is its neighbor in $T$. Let $T' := T-m$.
    Let
    \begin{align*}
        X'_{m-1} := \{v\in X_{m-1} \colon |N_{G}(v) \cap X_{m}| \ge 1\}
    \end{align*}
    We claim that $|X'_{m-1}| \ge |X_{m-1}|/2$.
    Indeed, suppose to the contrary that $|X'_{m-1}| < |X_{m-1}|/2$.
    %Then let $Y_{m-1} \subset X_{m-1}\setminus X'_{m-1}$ be a subset of size $|X_{m-1}|/2$.
     Then  we would have $|X_{m}||X_{m-1}\setminus X'_{m-1}| \ge |X_{m}||X_{m-1}|/2 \ge 2^{m-1}(\alpha/p)^2 \ge 2(\alpha/p)^2$, it follows from Lemma~\ref{lem:standard} (with $q=\sqrt{2}>1$) that $e(X_m, X_{m-1}\setminus X'_{m-1})>0$.
    This contradicts the fact that $e(X_m, X_{m-1}\setminus X'_{m-1}) =0$. Hence, $|X'_{m-1}| \ge |X_{m-1}|/2$.

    Now apply the induction hypothesis to the sets $X_1, \ldots, X_{m-2}, X'_{m-1}$, we obtain an embedding $f \colon T' \to G$ such that $f(i) \in X_i$ for $i\in [m-2]$ and $f(m-1) \in X'_{m-1}$. By the definition of $X'_{m-1}$, there exists $v \in X_{m}$ such that $\{f(m-1), v\} \in G$. Hence we can extend $f$ to get an embedding of $T$ to $G$ by setting $f(m) = v$.  This completes the proof of Lemma~\ref{LEMMA:embed-tree}.
\end{proof}
%%%%%%%%%%%%%%%%%%%%%%%%%%%%%%%%%%%%%%

Now we are ready to prove Theorem~\ref{THM:CFZ+tree}.

\begin{proof}[Proof of Theorem~\ref{THM:CFZ+tree}]
Let $G$ be a $(p, \alpha)$-jumbled graph with $\alpha < \frac{p^{d+1}n}{C}$ for an arbitrarily large constant $C > m4^m$. We will show that $G$ contains a copy of $F$. This implies that $\mathrm{exp}(F) \leq 1 - \frac{1}{2d+1}$. Indeed, if $C_1>0$ and $C_2 > (C_1C)^{\frac{2}{d+1}}$ and $G$ is $(p, \alpha)$-jumbled with $p > C_2 n^{-\frac{1}{2d+1}}$ and $\alpha < C_1\sqrt{pn}$, then a short calculation shows that $\alpha < \frac{p^{d+1}n}{C}$ and our result will imply the theorem.

Consider an ordering $v_1,v_2, \ldots, v_m$ of the vertices of $F$ and a $1 \leq k \leq m$ such that:
\begin{enumerate}[label=(\alph*)]
    \item  $F[\{v_{k}, \ldots, v_m\}]$ is a forest,
    \item for all edges $v_iv_j \in F$ with $i<j$, $i<k$, we have $N_{< i}(v_i) + N_{< i}(v_j) \leq 2d$, and
    \item for all edges $v_iv_j \in F$ with $k \leq i < j$, we have $N_{<k}(v_i) + N_{<k}(v_j) \leq 2d$.
\end{enumerate}
Let us denote $F[\{v_1, \ldots, v_{k-1}\}]$ by $F_1$ and $F[\{v_k, \ldots, v_m\}]$ by $F_2$. We will first embed a copy of $F_1$ using (b). At the same time, we will also ensure by (c), that the candidate sets for the vertices $v_k, \ldots, v_m$ are still large enough so that the forest $F_2$ can be embedded in them, thus giving an embedding of $F$.

Take a partition $V(G) = V_1 \cup \cdots \cup V_{m}$ such that $|V_i| \ge \frac{n}{2m}$ for all $i\in [m]$.

\begin{claim}\label{CLAIM:size-Xi}
Let $s\in [k-1]$. Then
there exist vertices $u_j \in V_j$ for all $j \in [s]$ such that $G[\{u_1, \ldots, u_{s}\}]$ contains a copy of $F[\{v_1, \ldots, v_{s}\}]$ and
\begin{align*}
V_{i,s}:=V_i \cap \left(\bigcap_{j \le s \colon v_iv_j\in F} N_G(u_j)\right)
\end{align*}
satisfies the inequality
$$|V_{i,s}| \geq \frac{p^{\left|N_{\le s}(v_i) \right|} |V_i|}{2^{s}}
\ge \frac{p^{\left|N_{\le s}(v_i) \right|} n}{m2^{m}}$$
 for all $i\in [s+1,m]$.
\end{claim}
\begin{proof}
For every $j\in [m]$ let $I_{j}:= \{\ell \in [j+1,m] \colon v_j v_{\ell} \in F\}$.
    The proof is by induction on $s$.
    For the base case $s=1$, first observe that
    \begin{align*}
        |V_1||V_j| \ge \left(\frac{n}{2m}\right)^2 \ge 4m \left(\frac{\alpha}{p}\right)^2
    \end{align*}
for all $j \in I_1$.
    Hence we can apply Lemma~\ref{lem:standard} to $V_1$ and $V_j$ for all $j \in I_1$ with $q=2$ to obtain a vertex $u_1 \in V_1$ such that $d(u_1, V_j) \ge p |V_j|/2$ for all $j \in I_1$.
    Now suppose that $s \ge 2$.
    Apply the induction hypothesis to get $u_i \in V_i$ for $i\in [s-1]$ such that $G[\{u_1, \ldots, u_{s-1}\}]$ contains a copy of $F[\{v_1, \ldots, v_{s-1}\}]$ and for every $i\in [s,m]$ the set $U_i := V_{i, s-1}$
    satisfies
\begin{align}\label{equ:Ui-size}
\left|U_i\right|
\geq \frac{p^{\left|N_{\le s-1}(v_i) \right|} |V_i|}{2^{s-1}}.
\end{align}

Observe that for every $j \in I_s$ we have
\begin{align}\label{equ:UiUj-induction}
    |U_s||U_j|
     \ge \frac{p^{\left|N_{\le s-1}(v_s) \right|} n}{m2^{m}} \cdot \frac{p^{\left|N_{\le s-1}(v_j) \right|} n}{m2^{m}}
     = \frac{p^{\left|N_{\le s-1}(v_s) \right|+ \left|N_{\le s-1}(v_j) \right|} n}{m^22^{2m}}
     & \ge \frac{p^{2d} n^2}{m^22^{2m}} \notag \\
     & > 2^m \left(\frac{\alpha}{p} \right)^2.
\end{align}
In the second last inequality we used  (b), and
in the last inequality we used the assumption that $\alpha \le p^{d+1}n/(m4^m)$.
So we may apply Lemma~\ref{lem:standard} to $U_s$ and $U_j$ for all $j \in I_s$ with $q=2$ and obtain $u_{s}\in U_{s}$ such that $d(u_s, U_j) \ge p|U_j|/2$ for all $j\in I_s$.
Now by (\ref{equ:Ui-size}), for every $i\in I_{s}$ we have
\begin{align*}
|V_{i,s}|\ge \left|U_i \cap N(u_s) \right|
\ge \frac{p}{2}\frac{p^{\left|N_{\le s-1}(v_i) \right|} |V_i|}{2^{s-1}}
= \frac{p^{\left|N_{\le s}(v_i) \right|} |V_i|}{2^{s}}.
\end{align*}
On the other hand, by (\ref{equ:Ui-size}), for every $i\in [s+1, m]\setminus I_{s}$, we have
\begin{align*}
|V_{i,s}|
= \left|U_i\right|
\geq \frac{p^{\left|N_{\le s-1}(v_i) \right|} |V_i|}{2^{s-1}}
\ge \frac{p^{\left|N_{\le s}(v_i) \right|} |V_i|}{2^{s}}.
\end{align*}
Finally, it is clear that $G[\{u_1, u_2, \ldots, u_{s}\}]$ contains a copy of $F[\{v_1, v_2, \ldots, v_{s}\}]$,  so the proof of the claim is complete.
\end{proof}

Applying Claim~\ref{CLAIM:size-Xi} with $s=k-1$ we obtain $u_j\in V_j$ for $j\in [k-1]$ such that $G[\{u_1, \ldots, u_{k-1}\}]$ contains a copy of $F_1$ and
\begin{align*}
   |V_{i,k-1}| \ge \frac{p^{\left|N_{< k}(v_i) \right|} n}{m2^{m}}
\end{align*}
for all $i\in [k,m]$.
Now that the first portion of the graph has been embedded, it remains only to embed a forest on the given candidate sets $X_i:=V_{i,k-1}$.
If we find  an embedding $f\colon F_2 \to G$ with $f(v_i) \in X_i$ for all $i\in [k,m]$,
%a copy $u_k, \ldots, u_m$ of $F_2$ such that $u_i \in X_i$ for all $i$,
then $G[\{u_1 ,\ldots,u_{k-1}, f(v_k), \ldots, f(v_m)\}]$ contains a copy of $F$.
Similar to (\ref{equ:UiUj-induction}), by (c) and Claim~\ref{CLAIM:size-Xi}, for every $\{v_i,v_j\}\in F_2$ we have
\begin{align*}
    |X_i||X_j|
     \ge \frac{p^{\left|N_{< k}(v_i) \right|} n}{m2^{m}} \cdot \frac{p^{\left|N_{< k}(v_j) \right|} n}{m2^{m}}
     = \frac{p^{\left|N_{< k}(v_i) \right|+ \left|N_{< k}(v_j) \right|} n}{m^22^{2m}}
     \ge \frac{p^{2d} n^2}{m^22^{2m}}
%    & \geq  2\cdot 100^m \left(\frac{\alpha}{p} \right)^{4-\frac{1}{d}\left(| N_{< k}(v_i)| + | N_{< k}(v_j)|  \right) } \\
     > 2^m \left(\frac{\alpha}{p} \right)^2.
\end{align*}
%In the last inequality we used the assumption $\alpha \le p^{d+1}n/(m4^m)$.
Applying Lemma~\ref{LEMMA:embed-tree} with $T = F_2$ and the sets $X_{k}, \ldots, X_{m}$, we know that such an embedding $f$ exists.
%we obtain an embedding $f\colon \{v_{k}, \ldots, v_{m}\} \to V(G)$ such that $f(v_i) \in X_i$ for all $i\in [k,m]$.
This completes the proof of Theorem~\ref{THM:CFZ+tree}.
\end{proof}

We remark that there are some graphs to which the precise statement of the above theorem cannot be applied in order to get a tight result - for example, odd cycles. However, the proof can be slightly adapted to deal with them. For odd cycles we take $k = 2$, so that $F[v_k, \ldots, v_m]$ is a path; this $F_2$ can then be embedded in a different way than in the general theorem above, in particular, using also the expansion properties of $(\alpha,p)$-jumbled graphs.

We now give a further generalization of Theorem~\ref{THM:CFZ+tree} where instead of partitioning the graph into two parts which are dealt with separately, we partition the graph into several parts.

For every graph $F$ on $m$ vertices, let $\hat{d}_{2}(F)$ denote the smallest number $d$ for which there exists an ordering $v_1,v_2, \ldots, v_m$ of $V(F)$ such that the following statements hold for some $\ell\in \mathbb{N}$ and $1 = k_1 < k_2 < \cdots < k_{\ell} < k_{\ell+1} = m$:
\begin{enumerate}[label=(\alph*)]
    \item $F[\{v_{k_s}, \ldots, v_{k_{s+1}-1}\}]$ is a forest for all $s\in [\ell]$,
    \item for all edges $v_iv_j \in F\setminus \left(\bigcup_{s=1}^{\ell}F[\{v_{k_s}, \ldots, v_{k_{s+1}-1}\}]\right)$ with $i< j$, we have $N_{< i}(v_i) + N_{< i}(v_j) \leq 2d$,  and
    \item for all $s\in [\ell]$ and for all edges $v_iv_j \in F[\{v_{k_s}, \ldots, v_{k_{s+1}-1}\}]$, we have $N_{<k_s}(v_i) + N_{<k_s}(v_j) \leq 2d$.
\end{enumerate}
It is clear that $\hat{d}_{2}(F) \le d_{2}(F)$ since in the definition of $d_{2}(F)$ we always let $\ell=m-1$ and $k_i = i$ for all $i\in [m]$.

\begin{theorem}\label{THM:CFZ+many-trees}
For every graph $F$ we have $\mathrm{exp}(F) \leq 1 - \frac{1}{2\hat{d}_{2}(F)+1}$.
\end{theorem}

 {\bf Remark.} One can extend Theorem~\ref{THM:CFZ+many-trees}  to get a counting result for $F$ in pseudorandom graphs that improves  Theorem~1.14 in~\cite{CFZ14} (by replacing $d_2(F)$ there with $\hat{d}_2(F)$ here).  This could result in some improvements for the corresponding Tur\'{a}n and Ramsey problems in pseudorandom graphs (see Theorems~1.4, 1.5, and 1.6 in~\cite{CFZ14}).

%%%%%%%%%%%%%%%%%%%%%%%%%%%%%%%%%
\section{$\{K_{2,3}, K_3\}$-free pseudorandom graphs}\label{SEC:K2-K23}
%Let $\mathbb{F}_{q}$ be a finite field such that $\mathrm{char}(\mathbb{F}_{q})\neq 3$.
In this section we present a construction of $\{K_{2,3}, K_3\}$-free pseudorandom graphs thereby proving Theorem~\ref{THM:K3-K2,3}.

Suppose that $F$ is a finite group and $S \subset H$ is a symmetric subset, i.e. $S = S^{-1}$.
Then the \textit{Cayley graph} $\mathrm{Cay}(H,S)$ is a graph on $F$ with edge set
\begin{align*}
    \left\{\{v, vs\} \colon v\in H \text{ and } s\in S\right\}.
\end{align*}

The spectrum, i.e. the eigenvalues of the adjacency matrix, of a Cayley graph can be represented by the characters of of $F$ (see e.g.~\cite{Lovasz75,Babai79}).
For our purpose, we only need the following result for the case that $F$ is an abelian group.

Recall that an abelian group $H$ can be represented as $H = \bigoplus_{i=1}^{k}\mathbb{Z}_{n_i}$ for some integers $k$ and $n_1, \ldots, n_{k}$.
For abelian groups we have a simple description of all the characters.
For each $\textbf{a} = (a_1, \ldots, a_{k}) \in \bigoplus_{i=1}^{k}\mathbb{Z}_{n_i}$ we have a character $\psi_{\textbf{a}} \colon H\to \mathbb{C}$ defined by
\begin{align*}
    \psi_{\textbf{a}}(h_1, \ldots, h_k) = \prod_{i=1}^{k} \omega_{n_i}^{a_i h_i},
\end{align*}
where $\omega_t = e^{2\pi i/t}$.

\begin{lemma}[see e.g.~\cite{Lovasz75,Babai79}]\label{LEMMA:spectrum-characters}
Suppose that $H = \bigoplus_{i=1}^{k}\mathbb{Z}_{n_i}$ is an abelian group.
Then the spectrum of the Cayley graph $\mathrm{Cay}(H,S)$ is
\begin{align*}
    \left\{\sum_{\textbf{s}\in S}\psi_{\textbf{a}}(\textbf{s}) \colon \textbf{a}\in H \right\}.
\end{align*}
\end{lemma}

Our main result is as follows.

\begin{theorem}\label{THM:K2-K23-Cayley}
Suppose that $p\neq 3$ is a prime number, $H =  \mathbb{Z}_{p}^2$, and $$S = \left\{(x,x^3)\colon x\in \mathbb{Z}_{p}\setminus \{0\}\right\}.$$
Then $\mathrm{Cay}(H,S)$ is a $\{K_{3}, K_{2,3}\}$-free $(n,d,\lambda)$-graph with $n = p^2$, $d = p-1$, and $\lambda \le 2 \sqrt{p}+1$.
\end{theorem}

We will use the following well known estimate of Weil in the proof of Theorem~\ref{THM:K2-K23-Cayley}.

Recall that the \textit{order} of a character $\chi$ is the smallest positive integer $d$
such that $\chi^d = \chi_0$, where $\chi_0$ is the trivial character.

\begin{theorem}[see e.g. {\cite[Theorem~13.3]{BOLLOBASBOOK}}]\label{THM:pseudo:Bollabas-cha-sum}
Let $\chi$ be a character of order $d > 1$.
Suppose that $f(X)\in \mathbb{F}[X]$ has precisely $m$ distinct zeros and it is not a $d$th power,
that is $f(X)$ is not the form $c\left(g(X)\right)^{d}$, where $c\in \mathbb{F}$ and $g(X)\in \mathbb{F}[X]$.
Then
\begin{align}
\left| \sum_{x\in \mathbb{F}}\chi\left(f(x)\right) \right| \le (m-1)\sqrt{p}.  \notag
\end{align}
\end{theorem}

\begin{proof}[Proof of Theorem~\ref{THM:K2-K23-Cayley}]
Let $G = \mathrm{Cay}(H,S)$. Let $n=p^2$. It is clear that the number of vertices in $G$ is $n$, and it follows from the definition of Cayley graphs that $G$ is $|S|$-regular.
Let $\lambda_1 \ge \cdots \ge \lambda_n$ be the eigenvalues of the adjacency matrix $A_{G}$ of $G$. Since $G$ is regular, we have $\lambda_1 = |S| = p-1$.

First we prove that $G$ is $K_3$-free.
Suppose to the contrary that there exist three vertices $\textbf{u}, \textbf{v}, \textbf{w} \in \mathbb{Z}_{p}^2$ that form a copy of $K_3$ in $G$.
Assume that $\textbf{v}-\textbf{u} = (a,a^3)$, $\textbf{w}-\textbf{v}=(b,b^3)$, and $\textbf{u}-\textbf{w}=(c,c^3)$.
Then
\begin{align}
a+b+c & = 0, \notag\\
a^3+b^3+c^3 &= 0. \notag
\end{align}
Therefore,
\begin{align}
0
= (a+b+c)(a^2+b^2+c^2) - (a^3+b^3+c^3)
& = ab(a+b)+ac(a+c)+bc(b+c) \notag\\
& = -3abc. \notag
\end{align}
Since $p\neq 3$, we must have $0\in \{a,b,c\}$, a contradiction.

Next we prove that $G$ is $K_{2,3}$-free.
It is equivalent to show that every pair of vertices $\{ \textbf{u}, \textbf{v} \} \subset \mathbb{Z}_{p}^2$ has at most two common neighbors.
Let $a = \textbf{u}_1-\textbf{v}_1$ and $b = \textbf{u}_2-\textbf{v}_2$.
A common neighbor of $\textbf{u}$ and $\textbf{v}$ implies that there exist $x,y\in \mathbb{Z}_{p}\setminus \{0\}$ such that
\begin{align}
y-x &= a, \mathrm{\ and} \notag\\
y^3-x^3 &= b. \notag
\end{align}
These two equations imply that $(x+a)^3 = x^3+b$, which simplifies to $3ax^2+3a^2x+a^3-b=0$.
Since $(a,b) \ne (0,0)$ and $p \ne 3$, this quadratic equation in $x$ has at most two solutions in $\mathbb{Z}_{p}\setminus \{0\}$.
Therefore, $\textbf{u}$ and $\textbf{v}$ have at most two common neighbors.

Finally, we prove that $|\lambda_i| \le 2\sqrt{p}$ for all $i\in [2,n]$.
By Lemma~\ref{LEMMA:spectrum-characters}, for every $i\in [n]$ there exists $(a_1, a_2) \in \mathbb{Z}_{p}^2$ such that
\begin{align}
\lambda_i
= \lambda_{(a_1,a_2)}
= \sum_{\textbf{s}\in S}\varphi_{(a_1,a_2)}(\textbf{s})
= \sum_{x=1}^{p-1}\omega_{p}^{a_1x+a_2x^3}  \notag
\end{align}
If $(a_1,a_2) = (0,0)$, then $\lambda_{(a_1,a_2)}  = |S| = p-1$, and this corresponds to $\lambda_1$.
So we may assume that $(a_1,a_2) \neq (0,0)$.

First, it is easy to see that the character $\chi \colon \mathbb{Z}_p \to \mathbb{C}^{\times}$ defined by
$\chi(\alpha) = \omega_{p}^{\alpha}$ for all $\alpha \in \mathbb{Z}_p$ has order $p$.
On the other hand, since $(a_1,a_2) \neq (0,0)$ and $p\neq 3$, the polynomial $f(X) = a_1 X + a_2 X^3$ is not of the form $c\left(g(X)\right)^p$ for any $c\in \mathbb{Z}_p$ and for any polynomial $g(X)$.
Therefore, it follows from Theorem~\ref{THM:pseudo:Bollabas-cha-sum} that
\begin{align*}
    \left|\sum_{x=1}^{p-1}\omega_{p}^{a_1x + a_2x^3} \right|
    = \left|\sum_{x=0}^{p-1}\omega_{p}^{a_1x + a_2x^3} - 1\right|
    \le \left|\sum_{x=0}^{p-1}\omega_{p}^{a_1x + a_2x^3}\right|+1
    \le 2\sqrt{p}+1.
\end{align*}
This implies that $|\lambda_i| \le 2\sqrt{p}+1$ for all $i\in [n]\setminus\{1\}$, and hence completes the proof of Theorem~\ref{THM:K2-K23-Cayley}.
\end{proof}
%%%%%%%%%%%%%%%%%%%%%%%%%%%%%%%%%

%%%%%%%%%%%%%%%%%%%%%%%%%%%%%%%%%%%%%%%%%%%%%%%%%%%%%
\section{$K_t$-free pseudorandom graphs}\label{SEC:Kt-free}
In this section we present a construction that is isomorphic to the construction of Bishnoi, Ihringer, and Pepe~\cite{BIP20}, and give a new proof
to Theorem~\ref{THM:BIP20}.

Denote by $\mathrm{PG}(t-1,q)$ the $(t-1)$-dimensional projective space over $\mathbb{F}_{q}$, i.e. $\mathrm{PG}(t-1,q) = \mathbb{F}_{q}^{t}/{\sim}$, where two vectors $\textbf{x}, \textbf{y} \in \mathbb{F}_{q}^{t}$ are equivalent under $\sim$ if there exists a non-zero element $a\in \mathbb{F}_{q}$ such that $\textbf{x} = a \textbf{y}$. For a vector $\textbf{x} \in \mathbb{F}_{q}^{t}$ we use $\llbracket\textbf{x}\rrbracket$ to denote its equivalence class in $\mathbb{F}_{q}^{t}/{\sim}$.
It is easy to see that the number of points in $\mathrm{PG}(t-1,q)$ is $\frac{q^t-1}{q-1} = (1+o(1))q^{t-1}$.
Recall that the dot-product $ \textbf{x}\cdot \textbf{y}$ of two vectors $\textbf{x}, \textbf{y} \in \mathbb{F}_{q}^{t}$ is defined as $\textbf{x}\cdot \textbf{y} = \sum_{i=1}^{t}x_iy_i$.
A point $\textbf{x}=(x_1,\ldots,x_t) \in \mathbb{F}_{q}^{t}$ is called
\begin{itemize}
\item
\textit{absolute} if $\textbf{x}\cdot \textbf{x} = 0$,
\item
\textit{square} if $\textbf{x}\cdot \textbf{x} = a^2$ for some $a\in \mathbb{F}_{q}$,
\item
\textit{non-square} if $\textbf{x}\cdot \textbf{x} \neq a^2$ for all $a\in \mathbb{F}_{q}$,
\end{itemize}
We use $X_{0}(t,q), X_{\Box}(t,q), X_{\boxtimes}(t,q)$ to denote the collection of all absolute points, square points, and non-square points in $\mathbb{F}_{q}^{t}$, respectively.
If $t$ and $q$ are clear from the context, we will omit them and use $X_{0}, X_{\Box}, X_{\boxtimes}$ for simplicity.
It is easy to see from the definition that if $\textbf{x} \in X_{0}$, $\textbf{x} \in X_{\Box}$, or $\textbf{x} \in X_{\boxtimes}$, then $\llbracket\textbf{x}\rrbracket \subset X_{0}$, $\llbracket\textbf{x}\rrbracket \subset X_{\Box}$, or $\llbracket\textbf{x}\rrbracket \subset X_{\boxtimes}$, respectively.

Recall that a \textit{character} of a group $H$ is a homomorphism $\psi \colon H \to \mathbb{C}^{\times}$, and
the \textit{quadratic character} $\chi(\cdot)$ of $\mathbb{F}_{q}$ is defined as
\begin{align}
\chi(x) =
\begin{cases}
0, & \text{if $x = 0$}, \\
1, & \text{if $x$ is a square}, \\
-1, & \text{if $x$ is a non-square}.
\end{cases} \notag
\end{align}

Let $\mathrm{AK}(t-1,q)$ be the graph whose vertices are non-absolute points of $\mathrm{PG}(t-1,q)$
and two vertices $\llbracket\textbf{x}\rrbracket$ and $\llbracket\textbf{y}\rrbracket$ are adjacent iff $\textbf{x}\cdot \textbf{y} = 0$.
Note that $\mathrm{AK}(2,q)$ is just the Erd\H{o}s--Renyi graph.
In~\cite{AK97}, Alon and Krivelevich proved that $\mathrm{AK}(t-1,q)$ is a $K_{t+1}$-free $(n,d,\lambda)$-graph with $n = (1+o(1))q^{t-1}$, $d = \Theta(n^{1-\frac{1}{t}})$, and $\lambda = \Theta(\sqrt{d})$.

Parsons~\cite{Parson76} proved that for $q$ odd, the induced subgraph of $\mathrm{AK}(2,q)$ on $X_{\boxtimes}/{\sim}$ is $K_3$-free.
%This is better than the Alon--Krivelevich bound for $r=3$.
Indeed, suppose to the contrary that there exist three distinct points $\llbracket\textbf{x}_1\rrbracket, \llbracket\textbf{x}_2\rrbracket, \llbracket\textbf{x}_3\rrbracket \in X_{\boxtimes}/{\sim}$ that induce a copy of $K_3$ in $\mathrm{AK}(2,q)$.
Then $\textbf{x}_1, \textbf{x}_2, \textbf{x}_3$ are pairwise orthogonal, which means that there exists a non-zero element $a\in \mathbb{F}_{q}$ such that $\textbf{x}_3 = a \textbf{x}_1 \times \textbf{x}_2$, where $\textbf{x}_1 \times \textbf{x}_2$ is the \textit{cross-product} of $\textbf{x}_1$ and $\textbf{x}_2$.
Therefore, we have
\begin{align*}
    \textbf{x}_{3}\cdot \textbf{x}_{3}
   = \left( a \textbf{x}_1 \times \textbf{x}_2\right) \cdot \left( a \textbf{x}_1 \times \textbf{x}_2\right)
   = a^2 \cdot \left(\textbf{x}_{1}\cdot \textbf{x}_{1}\right) \cdot \left(\textbf{x}_{2}\cdot \textbf{x}_{2}\right),
\end{align*}
where in the last equality we used the fact that $\textbf{x}_1 \cdot \textbf{x}_{2} = 0$.
Applying the quadratic character $\chi(\cdot)$ to both sides of the equation above we obtain
\begin{align*}
   -1 = \chi(\textbf{x}_{3}\cdot \textbf{x}_{3})
   = \chi(a^2) \cdot \chi(\textbf{x}_{1}\cdot \textbf{x}_{1}) \cdot \chi(\textbf{x}_{2}\cdot \textbf{x}_{2}) = 1 \cdot (-1) \cdot (-1)  =1,
\end{align*}
a contradiction. Therefore, the induced subgraph of $\mathrm{AK}(2,q)$ on $X_{\boxtimes}/{\sim}$ is $K_3$-free.

Our first aim in this section is to extend Parsons' proof to all odd integers $t\ge 4$.

First, the cross-product can be extended from $3$-dimensional space to $r$-dimensional space for every $r\ge 4$. We refer the reader to~\cite{Spivak65} and~\cite{Dittmer95} for the formal definition. Here we only recall some basic properties of the cross-product in $r$-dimensional space.

\begin{fact}[see e.g.~\cite{Dittmer95}]\label{FACT:cross-product}
Suppose that $\textbf{x}_1, \ldots, \textbf{x}_{t-1} \in \mathbb{F}_{q}^t$ are $t-1$ vectors. Then
\begin{enumerate}
    \item  $\textbf{x}_1 \times \cdots \times \textbf{x}_{t-1}$ is skew-symmetric and linear in each $\textbf{x}_i$,
    \item $\textbf{x}_1 \times \cdots \times \textbf{x}_{t-1}$ is a vector that is orthogonal to each of $\textbf{x}_1, \ldots, \textbf{x}_{t-1}$,
    \item $\textbf{x}_1 \times \cdots \times \textbf{x}_{t-1} = 0$ iff $\textbf{x}_1, \ldots, \textbf{x}_{t-1}$ are linearly dependent.
\end{enumerate}
\end{fact}

A proof of the following theorem can be found in~\cite{Dittmer95}.

\begin{theorem}[see e.g.~\cite{Dittmer95}]
Let $\textbf{x}_1,\ldots,\textbf{x}_{t-1}, \textbf{y}_1,\ldots, \textbf{y}_{t-1}\in \mathbb{F}_{q}^t$.
Then
\begin{align}\label{EQU:pseudo:cross-product}
(\textbf{x}_1\times \cdots \times \textbf{x}_{t-1}) \cdot (\textbf{y}_1\times \cdots \times \textbf{y}_{t-1})
& = \mathrm{det}
\begin{bmatrix}
    \textbf{x}_1\cdot \textbf{y}_1 & \dots & \textbf{x}_1\cdot \textbf{y}_{t-1} \\
    \vdots & \ddots & \vdots \\
    \textbf{x}_{t-1}\cdot \textbf{y}_1 &  \dots & \textbf{x}_{t-1}\cdot \textbf{y}_{t-1}
    \end{bmatrix}.
\end{align}
\end{theorem}

Our main result in this section is as follows.

\begin{theorem}\label{THM:pseudo:r-odd-Kr-free}
Suppose that $q$ is an odd prime power and $t\ge 3$ is an odd integer. Then
the induced subgraph of $\mathrm{AK}(t-1,q)$ on $X_{\boxtimes}/{\sim}$ is a $K_{t}$-free $(p,\alpha)$-jumbled graph with $p=\Theta\left(n^{-\frac{1}{t-1}}\right)$ and $\alpha = \Theta(\sqrt{np})$, where $n = |X_{\boxtimes}/{\sim}|$.
\end{theorem}
\begin{proof}
Suppose to the contrary that there exists a set $S$ of $t$ distinct points $\llbracket\textbf{x}_1\rrbracket,\ldots, \llbracket\textbf{x}_t\rrbracket \in X_{\boxtimes}/{\sim}$ such that
the induced subgraph of $\mathrm{AK}(t-1,q)$ on $S$ is complete.
Then it follows from Fact~\ref{FACT:cross-product} that there exists a nonzero element $a\in \mathbb{F}_{q}$ such that $\textbf{x}_{t} = a \textbf{x}_{1}\times \cdots \times \textbf{x}_{t-1}$.
Therefore, by (\ref{EQU:pseudo:cross-product}), we have
\begin{align}
\textbf{x}_{t} \cdot \textbf{x}_{t}
& = (a \textbf{x}_{1}\times \cdots \times \textbf{x}_{t-1}) \cdot (a\textbf{x}_{1}\times \cdots \times \textbf{x}_{t-1})  \notag \\
& =a^2 \cdot \mathrm{det}
\begin{bmatrix}
    \textbf{x}_1\cdot \textbf{x}_1 & \dots & \textbf{x}_1\cdot \textbf{x}_{t-1} \\
    \vdots & \ddots & \vdots \\
    \textbf{x}_{t-1}\cdot \textbf{x}_1 &  \dots & \textbf{x}_{t-1}\cdot \textbf{x}_{t-1}
    \end{bmatrix}
= a^2 \cdot \prod_{i=1}^{t-1} \left(\textbf{x}_i \cdot \textbf{x}_i\right). \notag
\end{align}
In the last equality we used the fact that $\textbf{x}_i\cdot \textbf{x}_j = 0$ for all $i\neq j$.
Applying the quadratic character $\chi(\cdot)$ to both sides of the equation above, we obtain
\begin{align}
-1
= \chi\left(\textbf{x}_{t} \cdot \textbf{x}_{t} \right)
= \chi\left(a^2 \cdot \prod_{i=1}^{t-1} \left(\textbf{x}_i \cdot \textbf{x}_i\right)\right)
= \chi(a^2) \cdot \prod_{i=1}^{t-1} \chi\left(\textbf{x}_i \cdot \textbf{x}_i\right)
=1\cdot (-1)^{t-1}
= 1,  \notag
\end{align}
a contradiction.
\end{proof}%theorem

For the case that $t\in \mathbb{N}$ is even we use a different argument.

\begin{proposition}\label{PROP:pseudo:neighborhood-induced-gp}
For every vertex $v\in \mathrm{PG}(t,q)$ the induced subgraph of $\mathrm{AK}(t,q)$ on $N(v)$
is isomorphic to $\mathrm{AK}(t-1,q)$.
\end{proposition}
\begin{proof}
Let $\textbf{e}_1, \ldots, \textbf{e}_{t}$ be the standard orthonormal basis of the $t$-dimensional space $\mathbb{F}_{q}^{t}$.
Fix a vector $\textbf{v}\in \mathbb{F}_{q}^{t+1}\setminus X_0$ and let $\textbf{e}'_1, \ldots, \textbf{e}'_{t}$ be an orthonormal basis of the $t$-dimensional space $\textbf{v}^{\bot}$,
where $\textbf{v}^{\bot} = \left\{\textbf{w} \in \mathbb{F}_{q}^{t+1} \in \colon \textbf{w} \cdot \textbf{v} = 0\right\}$.
Define the map $\phi\colon \mathbb{F}_{q}^{t} \to \textbf{v}^{\bot}$ by sending $\sum_{i=1}^{t}a_i\textbf{e}_i$ to $\sum_{i=1}^{t}a_i\textbf{e}'_i$.
Clearly, the map $\phi$ is linear and induces a bijection between $\mathbb{F}_{q}^{t}/{\sim}$ and $\textbf{v}^{\bot}/{\sim}$. Moreover, $\phi$ sends absolute points to absolute points.
Now suppose that $\textbf{u}_1 = \sum_{i=1}^{t}a_i\textbf{e}_i$ and $\textbf{u}_2 = \sum_{i=1}^{t}b_i\textbf{e}_i$ are two distinct points in $\mathbb{F}_{q}^{t}$.
%and $u_1\cdot u_2 = 0$, i.e. $\sum_{i=1}^{r}x_iy_i = 0$.
Then
\begin{align}
\psi(\textbf{u}_1)\cdot \psi(\textbf{u}_2)
= \left(\sum_{i=1}^{t}a_i\textbf{e}'_i\right)\cdot \left(\sum_{i=1}^{t}b_i\textbf{e}'_i\right)
= \sum_{i=1}^{t}a_ib_i
= \left(\sum_{i=1}^{t}a_i\textbf{e}_i\right)\cdot \left(\sum_{i=1}^{t}b_i\textbf{e}_i\right)
= \textbf{u}_1\cdot \textbf{u}_2. \notag
\end{align}
This implies that the map $\phi$ preserves the orthogonality of two vectors, and hence, it sends edges (resp. non-edges) in $\mathrm{AK}(t,q)$ to an edge (resp. non-edge) in the induced subgraph of $\mathrm{AK}(t,q)$ on $\textbf{v}^{\bot}$.
Therefore, $\phi$ induces an isomorphism between $\mathrm{AK}(t-1,q)$ and the induced subgraph of $\mathrm{AK}(t,q)$ on $N(\textbf{v})$.
\end{proof}%prop

\begin{lemma}\label{LEMMA:pseudo:vertex-neighbor-partition}
Suppose that $\alpha \in (0,1)$ is a constant and $V_1 \subset \mathrm{PG}(t,q)$ is a subset of size $\alpha \cdot |\mathrm{PG}(t,q)|$ in the graph $\mathrm{AK}(t,q)$.
Then there exists a vertex $v\in V_1$ such that
\begin{align*}
    \frac{|N(v)\cap V_1|}{|N(v)|} \ge  (1-o_{q}(1))\alpha.
\end{align*}
\end{lemma}
\begin{proof}
Suppose to the contrary that there exists an absolute constant $\epsilon>0$ such that
$\frac{|N(v)\cap V_1|}{|N(v)|} \le  (1-\epsilon)\alpha$ for all $v\in V_1$ and for all $q$.
Choose $q$ to be sufficiently large.
Let $n = |\mathrm{PG}(t,q)|$ be the number of vertices in $\mathrm{AK}(t,q)$,
and let $d$ be the degree of $\mathrm{AK}(t,q)$.
Then, it follows from our assumption that
\begin{align}\label{EQU:pseudo:e-V1-upper-bound}
e(V_1)
= \frac{1}{2}\sum_{v\in V_1}|N(v)\cap V_1|
& \le \frac{1}{2} (1-\epsilon)\alpha |N(v)||V_1| \notag\\
& =  \frac{1}{2} (1-\epsilon)\frac{d}{n}|V_1|^2 \notag\\
& = \frac{1}{2}\frac{d}{n}|V_1|^2 - \frac{\epsilon}{2}\frac{d}{n}|V_1|^2
= \frac{1}{2}\frac{d}{n}|V_1|^2 - \frac{\epsilon \alpha}{2} d |V_1|. \notag
\end{align}
Since $\frac{\epsilon \alpha}{2} d \gg \sqrt{d}$, this contradicts the fact that $\mathrm{AK}(t,q)$ is $(\frac{d}{n}, \Theta(\sqrt{d}))$-jumbled.
\end{proof}%lemma

Now we are ready to prove Theorem~\ref{THM:BIP20} for even $t$.
Our construction will be an induced subgraph of $\mathrm{AK}(t,q)$ on a subset of the neighborhood of a vertex.

\begin{proof}[Proof of Theorem~\ref{THM:BIP20} for even $t$]
Let $t \in \mathbb{N}$ be an even number.
Let $V$ denote the vertex set of $\mathrm{AK}(t,q)$.
Let $V_1 = X_{\boxtimes}/{\sim}$.
Since $|V_1| = (1/2+o(1)) |\mathrm{PG}(t,q)|$,
by Lemma~\ref{LEMMA:pseudo:vertex-neighbor-partition}, there exists a vertex $\llbracket\textbf{v}\rrbracket\in V_1$ such that
$|N(\llbracket\textbf{v}\rrbracket)\cap V_1| \ge \left(\frac{1}{2}-o(1)\right)|N(\llbracket\textbf{v}\rrbracket)|$.
Let $U = N(\llbracket\textbf{v}\rrbracket)\cap V_1$.
By Proposition~\ref{PROP:pseudo:neighborhood-induced-gp},
the induced subgraph of  $\mathrm{AK}(t,q)$ on the set $N(\llbracket\textbf{v}\rrbracket)$ is isomorphic to $\mathrm{AK}(t-1,q)$, which is $(p, \alpha)$-jumbled with $p = \Theta(m^{-\frac{1}{t-1}})$ and
$\alpha  = \Theta(\sqrt{mp})$, where $m = |\mathrm{PG}(t-1,q)|$.
On the other hand, by Theorem~\ref{THM:pseudo:r-odd-Kr-free},
the induced subgraph of $\mathrm{AK}(t,q)$ on the set $X_{\boxtimes}/{\sim}$ is $K_{t+1}$-free.
Therefore, the induced subgraph of  $\mathrm{AK}(t,q)$ on the set $U$ is $K_{t}$-free.
This proves Theorem~\ref{THM:BIP20} for even $t$.
\end{proof}

%%%%%%%%%%%%%%%%%%%%%%%%%%%%%%%%%%%%%%
\section{Acknowledgment}
We would like to thank Anurag Bishnoi, Ferdinand Ihringer, and Thang Pham for their insightful comments.
%%%%%%%%%%%%%%%%%%%%%%%%%%%%%%%%%%%%%%%%%%%%%%%%%
\bibliographystyle{abbrv}
\bibliography{pseudorandom}
\end{document}